\documentclass[sigconf]{acmart}

\AtBeginDocument{%
  \providecommand\BibTeX{{%
    \normalfont B\kern-0.5em{\scshape i\kern-0.25em b}\kern-0.8em\TeX}}}

\copyrightyear{2021}
\acmYear{2021}
\setcopyright{acmcopyright}
\acmConference[ISSAC '21] {Proceedings of the 2021 International Symposium on Symbolic and Algebraic Computation}{July 18--23, 2021}{Virtual Event, Russian Federation.}
\acmBooktitle{Proceedings of the 2021 International Symposium on Symbolic and Algebraic Computation (ISSAC '21), July 18--23, 2021, Virtual Event, Russian Federation}
\acmISBN{}
\acmDOI{}

\usepackage{latexsym,amsxtra}
\usepackage{amsfonts}
\usepackage{verbatim}
\usepackage{amsmath}
\usepackage{graphicx}

\usepackage{balance}

\numberwithin{equation}{section}

\def\la{\langle}
\def\ra{\rangle}

\def\sto{\ {\overset{*}{\to}\ }}

\newcommand{\KK}{\ensuremath{\mathbf k}}  
\newcommand{\N}{\ensuremath{\mathbb N}}

\newcommand{\s}{\ensuremath{\mathbb S}}
\newcommand{\Z}{\ensuremath{\mathbb Z}}

\def\dim{\mbox{dim\,}}

\newcommand \f{\mathcal{F}}         

\def\height{\mbox{\bf ht\,}} 

\def\length{\mbox{length\,}}

\def\rtree{\prec} 
\def\reqtree{\preceq} 

\def\Zpl{\mbox{\bf Z}_+}

\def\ass{\: {\mathcal A}ss \:}

\def\p{ {\mathcal P} }

\begin{document}
\fancyhead{}
\title{Wilf classes of non-symmetric operads}


\author{Andrey T. Cherkasov}
\email{atcherkasov@hotmail.com}
\affiliation{%
	\institution{HSE University}
	\streetaddress{Myasnitskaya str. 20}
	\city{Moscow} 
	\country{Russia} 
	\postcode{101000}
}

\author{Dmitri Piontkovski}
\authornote{Corresponding author.}
\email{dpiontkovski@hse.ru}
\affiliation{%
	\institution{HSE University}
  \streetaddress{Myasnitskaya str. 20}
  \city{Moscow} 
  \country{Russia} 
  \postcode{101000}
}


%
%
%
%
%
%

\renewcommand{\shortauthors}{Cherkasov and Piontkovski}

\begin{abstract}
Two operads are said to belong to the same Wilf class if they have the same generating series. We discuss possible Wilf classifications of non-symmetric operads with monomial relations. As a corollary, this would give the same classification for the operads with a finite Groebner basis.
 
Generally, there is no algorithm to decide whether two finitely presented operads belong to the same Wilf class. Still, we show that if an operad has a finite Groebner basis, then the monomial basis of the operad forms an unambiguous context-free language. Moreover, we discuss the deterministic grammar which defines the language. The generating series of the operad can be obtained as a result of an algorithmic elimination of variables from the algebraic system of equations defined by the Chomsky–Sch{\"u}tzenberger enumeration theorem. We then focus on the case of binary operads with a single relation. The approach is based on the results by Rowland on pattern avoidance in binary trees. We improve and refine  Rowland's calculations and empirically confirm his conjecture. Here we use both the algebraic elimination and the direct calculation of formal power series from algebraic systems of equations. Finally,  we discuss the connection of Wilf classes with algorithms for calculation of the Quillen homology of operads.
\end{abstract}

\ccsdesc[500]{Mathematics of computing~Generating functions}
\ccsdesc[500]{Computing methodologies~Algebraic algorithms}

\keywords{nonsymmetric operads,
	unambigous grammar,
	Wilf class,
	Groebner bases in operads,
	tree pattern avoidance,
	system of algebraic equations,
	generating function,
	operad homology}

\maketitle

\section{Introduction}

An algebraic operad (either symmetric or not) is the union of a sequence of vector spaces, $P = P_0 \cup P_1 \cup \dots$ 
So, the first invariant of the operad is the sequence of dimensions $\dim P_0, \dim P_1, \dots$
either per se or in the form of the generation function, which is called the generating series of the operad. 

Suppose that we know a Groebner basis of the operad. Note that the generating series is equal to the one of the associated monomial operad. 
The generating series of the monomial operad is equal to the generating series of the set of trees avoiding certain patterns, see~\cite{kp}. 
These patterns correspond to the leading monomials of the Groebner basis.
Two sets of tree patterns are called {\em Wilf equivalent} if the corresponding sets of avoiding trees have the same generating series~\cite{rowland2010pattern, dots2012}. Similarly, we call two operads (with the same generating sets)  Wilf equivalent if their generating series coincide. 

We see that two operads belong to the same Wilf equivalence class if  the sets of leading monomials of their Groebner bases (with respect to the same generators, for arbitrary orderings) are Wilf equivalent as the sets of tree patterns. Since both operads are defined via the same sequence of vector space dimensions $\dim P_0, \dim P_1, \dots$, in some (weak) sense, they can be considered as flat deformations of each other.  Under some additional conditions (such as Koszulity) all operads of the same Wilf classe have the same homological invariants. Wilf equivalence can be considered as the weakest version of the isomorphism of operads. 

In this paper, we focus on non-symmetric operads. 
Is there an algorithm to determine if two finitely presented operads (defined by finite lists of generators 
and relations over a computable\footnote{We call a field computable if there exist finite presentations for all 
its elements and  algorithms for the arithmetic operations.} field) are Wilf equivalent? Generally, there is not. Such a general algorithm does not exist even if we assume 
that the operads are quadratic (see Proposition~\ref{prop:non-exist} below).

A number of important operads, however, have finite Groebner bases, so that they belong to the Wilf classes of finitely presented monomial operads. 
The generating series of the finitely presented monomial operads are algebraic functions~\cite{kp}. 
To determine the Wilf class, one needs 
to (1) construct a system of algebraic equations which defines the generating series and (2)  find the generating series from the system. The problem (2) is a standard problem of the algebraic elimination theory; its solution is based on the 
Groebner bases theory. Three algorithms for the problem (1) are discussed in~\cite{kp} and one algorithm is given  in~\cite{giraudo2020tree};
the first two  are generalized versions of the algorithms by Rowland~\cite{rowland2010pattern} for 
enumeration pattern avoiding binary trees. 

Here we go further by proving that such a series is $\N$-algebraic, that is, it is equal to the generating function of some unambiguous context-free language. 
It follows from 
\begin{theorem}[Corollary~\ref{prop:operad_cf}]
	\label{th:operads_cf_intro}
	Let $P$ be a finitely generated non-symmetric operad having a finite Groebner basis of relations.
Then its natural monomial basis forms a deterministic context-free language. 
\end{theorem}
In Section~\ref{sec:operads_cf}, we discuss the algorithm to construct the unambiguous grammar defining the language. 
Then the Chomsky--Sch{\"u}tzenberger enumeration theorem gives a way
to construct the system of algebraic equations from the grammar. 
We also see that the systems obtained by the methods of \cite[Section~2.2]{kp} give equivalent systems. 
 
Unfortunately, this algorithm in some cases produces different systems for two operads of the same Wilf class (see example in Section~\ref{sec:one_rel_rowland}). So, we should provide an elimination of variables. We provide a number of computer experiments in the case of operads generated by one binary operation with one monomial relation. Such operads encode the varieties of binary (non-associative) algebras which have been widely studied for decades. If we fix the degree (= the arity) of the relation, then Wilf classes for operads correspond to the Wilf classes of the binary tree  patterns with a given fixed number of leaves. The number of Wilf classes for the relation of arity $n$ is then equal to the value $A(n)$ of the sequence A161746 in Sloane's  On-Line Encyclopedia of Integer Sequences (https://oeis.org/A161746).
Such Wilf classes have been studied by Rowland~\cite{rowland2010pattern}.

We particularly improve the results by Rowland. That is, we calculate the term $A(8) = 43$ (the previous terms are calculated in~\cite{rowland2010pattern}) and give an upper bound and conjectural values for the terms $A(9)$--$A(12)$. The exact value $A(8) =43$ is based on the algebraic elimination of the variable. It is provided using the Wolfram Mathematica tool based on the Groebner elimination.  However,  we have found that the value $n=8$ is the practical limit for recent Groebner algorithms performance (at least on a standard PC), that is, for $n=9$ some cases cannot be practically treated.   We have developed a Python package that calculates the values $\dim P_t$ step-by-step using recurrent relations from the system of algebraic equations. After some value of $t$, the number of truncated generating functions looks stable. We conjecture that the number of different truncated generating functions is the number of Wilf classes. At least, these are the lower bounds for the number of Wilf classes. The resulted values are given in Section~\ref{sec:one_rel_rowland}.
  
One surprising observation by Rowland is that, for all patterns with at most 7 leaves,  the Wilf class determines the so-called enumeration generating  functions which encode the number of occurrences of the pattern inside all trees.
We show that this result is connected to some homological construction related to operads. We confirm the Rowland observation in the new cases  and discuss the connection with the Quillen homology of operads in the final Section~\ref{sec:homology}. This gives new conjectures and new algorithms for the estimation of the Betti numbers of operads. 
  
The paper is organized as follows. After a background in Section~\ref{sec:back}, we prove the general algorithmic indecidability for Wilf classes in Section~\ref{sec:indecidable}. In Section~\ref{sec:operads_cf} we show that the finitely presented monomial operads are described by unambiguous context-free languages. We also discuss the implications for their generating functions.  In  Section~\ref{sec:one_rel_rowland}, we discuss and particularly improve some of the results of Rowland on pattern avoidance in binary trees.  This gives some results and conjectures concerning one-relator binary operads. In Section~\ref{sec:homology}, we briefly discuss the homological interpretation of these results and conjectures.

\section{Background: operads, generating series, and Groebner bases }

\label{sec:back}

\subsection{Operads and generating functions}

\label{subs:oper_defs}

For the details on algebraic operads we refer the reader to the monographs~\cite{oper} and~\cite{Loday_Valet}; see also the textbook~\cite{db}.

We consider multioperator linear algebras over a field $\KK$ of zero characteristics.
A {\em variety} $W$ of $\KK$--linear algebras\footnote{We consider here varieties of algebras without constants, with the identity operator  and without other unary operations}
is the set of algebras admitting a given collection ({\em signature}) of multilinear operations 
$\Omega$ and satisfying some fixed polynomial identities. We assume that $\Omega$ is a finite union of finite sets $\Omega = \Omega_2 \cup \dots \cup \Omega_k$ where the elements $\omega$ of  $\Omega_t$ act on each algebra $A \in W$ as $t$-linear operations, $\omega:A^{\otimes t} \to A$. 
The variety is defined by two sets, the signature $\Omega$ and a set of defining identities $R$. By the linearization process, one can assume that $R$ consists of multilinear identities. 
Consider the free algebra $F^W (x)$ on a countable set of indeterminates $x= \{ x_1, x_2, \dots \}$. Let $\p_n \subset F$ be the subspace consisting of all multilinear generalized homogeneous polynomials on the variables $\{ x_1,
\dots, x_n \}$, that is, $\p_n$ is the component $F^W (x) [1, \dots, 1, 0, 0,\dots]$ with respect to the $\Z^\infty$ grading by the degrees on $x_i$.

\begin{definition}
	Given such a variety $W$,  the sequence $\p_W = \p := \{ \p_1,
	\p_2, \dots \}$ of the vector subspaces of $F^W (x)$ is called	an {\it operad}\footnote{More precisely, symmetric connected		$\KK$--linear operad with identity.}. 
\end{definition}

The $n$-th component  $\p_n$  may be identified with the set of
all derived $n$-linear operations on the algebras of $W$; in
particular,  $\p_n$ carries the natural structure of a
representation of the symmetric group $S_n$. Such a sequence $Q =
\{ Q(n) \}_{n \in \Z}$ of representations $Q(n)$ of the symmetric
groups $S_n$ is called an {$\s$--module}, so that an operad
carries a structure of $\s$-module with $\p_n = \p(n)$.  The compositions of
operations (that is, a substitution of an argument $x_i$ by a
result of another operation with a subsequent monotone re-numbering of the inputs to avoid repetitions) gives natural maps of
$S_*$-modules $\circ_i :  \p(n)\otimes \p(m) \to \p(n+m-1)$. Note
that the axiomatization of these operations gives an abstract
definition of operads, see~\cite{oper}.

The signature $\Omega$ can be considered as a sequence of subsets of $\p$ with $\Omega_n \subset \p_n$. 
Then $\Omega$ generates the operad $\p$ up to the $\s$--module structure and the  compositions $\circ_i$ so that it is called the {\em set of generators} of the operad.  

More generally, the  $\s$-module $X$ 
generated by $\Omega$ is called the (minimal) {\em module of generators} of the operad $\p$.
It can be also defined independently of $\Omega$ as $X = \p_+/(\p_+\circ \p_+)$ where $\p_+ = \p_2\cup \p_3 \cup \dots$ and $\circ$
denotes the span of all compositions of two $\s$-modules.
Then one can define a variety $W$ corresponding to a (formal) operad $\p$ by picking a set $\Omega$
of generators of $X$ to be the signature and considering all relations in $\p$ as the defining identities of the variety,
so that the variety $W$ can be recovered by $\p$ ``up to a change of variables''. One can consider the algebras from $W$ as vector spaces $V$ with the actions $ \p(n): V^{\otimes
	n} \to V$ compatible with compositions  and the $\s$-module structures, so that the algebras of $W$ are recovered by $\p$ up to  isomorphisms. 

Given an $\s$-module $X$, one can also define a {\em free operad} $\f(X)$ generated by $X$ as the span of all possible compositions of a basis of $X$ modulo the action of symmetric groups. For example, the free operad $\f(\s\Omega)$
on the free $\s$-module $\s\Omega$ corresponds to the variety of all algebras of signature $\Omega$. 

Suppose that the defining identities $R$ of the variety 
$W$ can be chosen in such a way that for each 
$$
f(x_1, \dots , x_n) = \sum_i \alpha_i f_i(x_1, \dots , x_n) \in R,
$$
where $\alpha_i \in \KK$ and $f_i$ are monomials (that is, the compositions of the operations from $\Omega$), in all monomials $f_i$ the variables  $x_1, \dots , x_n$ occur in the same relative order. A standard example is the variety of associative algebras, see below.
Then one can associate to $W$ a simpler {\em non-symmetric operad}. 

Generally, a non-symmetric operad is a union $P = P_1\cup P_2\cup \dots$
with the compositions $\circ_i$ as above but without the actions of the symmetric groups. To distinguish them, we refer to the operads defined above as {\em symmetric}. Each symmetric operad can be considered as a non-symmetric one.
To each non-symmetric operad $P$ one can assign a symmetric operad  $\p$ where $\p_n = S_n P_n$ is a free $S_n$ module generated by $P_n$. Then $\p$ is called a  {\em symmetrization}  of $P$. In particular, here $\dim \p_n = n! \dim P_n$.

An $n$-th codimension of a variety $W$ is just the dimension of the
respective operad component: $c_n(W) = \dim_k \p_n$ for $\p = \p_W$.
We consider both exponential and ordinary generating functions for this sequence:
\begin{equation}
	\label{eq::E::gen::ser}
	E_{\p} (z) := \sum_{n \ge 1} \frac{\dim \p(n)}{n!} z^n , G_{\p} (z) := \sum_{n \ge 1} {\dim \p(n)} z^n .
\end{equation}
For example, if $\p$ is  a symmetrization of a non-symmetric operad $P$ then 
$E_{\p} (z) = G_P (z)$. By a {\em generating series} of a symmetric operad $\p$ we mean the exponential generating function $\p(z) = E_{\p} (z)$.
In contrast, for a non-symmetric operad $P$ its generating series is defined as the ordinary generating function $P(z) =  G_{P} (z)$. In the case of varieties, both the  ordinary and exponential versions of the codimension series are studied. 

If the set $\Omega $ is finite then the series $\p (z)$ defines an analytic function
in the neighborhood of zero. For example, the non-symmetric operad $\mathrm{Ass}$ of associative algebras 
is the operad defined by one binary operation $m$ (multiplication) subject to the relation 
\linebreak $m(m(x_1,x_2),x_3)) = m(x_1,m(x_2,x_3))$ which is the associativity identity. Its $n$-th component consists of the only equivalence class of all arity $n$ compositions of $m$ with itself modulo the relation, so that 
$
\mathrm{Ass}(z) = G_{\mathrm{Ass}} (z) = \frac{z}{1-z}.
$
Its symmetrization is the symmetric operad $\ass$ generated by two operations $m(x_1,x_2)$ and $m'(x_1,x_2)=m(x_2,x_1)$
with the $S_2$ action $(12) m' = m$ 
subject to all the relations of the form $m(m(x_i,x_j),x_k)) = m(x_i,m(x_j,x_k))$. 
By the above, we have $E_{\ass}(z) = \ass (z) = \mathrm{Ass}(z)= G_{\mathrm{Ass}} (z)$, so that $\dim \ass_n = n!$. 

\label{sec: associativity_oper}

\subsection{Monomial bases and Groebner bases in operads}

\label{sec:groeb}

As the motivation for studying the monomial operads in the operad theory comes often via Groebner bases, we 
briefly recall here some basic facts of  the theory of Groebner bases (essentially, in non-symmetric operads). The reader who is interested in monomial operads  only can skip this subsection.

The Groebner bases in (shuffle) operads are introduced in~\cite{dk}; see also~\cite{Loday_Valet}. The Groebner bases for non-symmetric operads are discussed in~\cite{db}. 

Fix a discrete set $\Omega$ of generators of a  non-symmetric free operad.
A {\em nonsymmetric monomial} is a multiple composition of operations from $\Omega$. We refer to them simply as monomials.
Each monomial is represented by a rooted planar tree with 
internal vertices labelled by operations. We assume that the edges of the tree lead from the root to the leaves which are free edges. 

All  monomials (including the empty monomial corresponding to the identical operation) form a linear basis of the free non-symmetric operad generated by $\Omega$. 
Two  monomials are called isomorphic if they are isomorphic as labelled trees.  A monomial $P$ is {\em divisible} by a monomial $Q$ if $Q$ is isomorphic to a submonomial of $P$ where 'submonomial' means a labelled subtree.


There are families of orderings on the sets of non-symmetric monomials which are compatible with the corresponding compositions. This defines the notion of the leading term of an element of a free operad and leads to the rich Groebner bases theory. The theory includes a version of the Buchberger algorithm~\cite{dk} and even the triangle lemma~\cite{db}.
We will call the Groebner basis of the relation ideal of an operad $\p$ simply the Groebner basis of $\p$.
Whereas a general operad could have no finite Groebner basis, a number of important operads (including the operad of associative algebras and their generalized versions) admit such bases.

The first known implementation of Groebner base algorithms for an operad is the Haskell package {\sf Operads}~\cite{dv}. Its slightly improved version with some bugs fixed by A. Lando 
can be downloaded at  https://github.com/Dronte/Operads .
A new Haskell package for operadic Groebner bases  (due to Dotsenko and Heijltjes) has been recently published at http://irma.math.unistra.fr/~dotsenko/Operads.html . 

%
%

\subsection{Growth and generating series for operads with finite Groebner bases} 

\label{sec:we}


The generating series of an operad with a known Groebner basis is equal to the generating series of the corresponding {\em monomial} operad, that is, a shuffle operad or a non-symmetric operad whose relations are the leading monomials of the corresponding Groebner basis. 
The dimension of the $n$-th component of a monomial operad is equal to the number of the monomials of arity $n$ which are not divisible by the monomial relations of the operad. In this section, we consider monomial operads only.

For such an operad, the calculation of the dimensions of its components is a purely combinatorial problem of the enumeration of the labelled trees which do not contain a subtree isomorphic to a relation as a submonomial (a pattern avoidance problem for labelled trees), see~\cite{dk-pattern}. Unfortunately, this problem is too hard to be treated  in its full generality. In this section  we discuss some partial methods based on the results of~\cite{kp}.
Note that the generating series of general monomial quadratic nonsymmetric operads were first discussed  by Parker~\cite{parker1993combinatorics} in other terms.

First, let us discuss a simpler case of non-symmetric operads.

\begin{theorem}[\cite{kp}, Th.~{2.3.1}]
	\label{th-nonsym-intro}
	The ordinary generating series of a non-symmetric operad with finite Gr\"obner
	basis is an algebraic function.
\end{theorem}

One of the methods for finding the algebraic equation for the generating series of a  non-symmetric operad $P$ defined by a finite number of monomial relations $R$ is the following. We consider the  monomials (called stamps)  of the level  less than the maximal level of an element of $R$ which is nonzero in $P$. For each stamp $m=m_i$, we consider the generating function $y_i(z)$ of the set of all nonzero monomials which are left divisible by $m_i$ and are not left divisible by $m_t$ with $t<i$. Then the sum of all $y_i(z)$ 
is equal to $P(z)$. The divisibility relation on the set of all stamps leads to a system of $N$ equations of the form: 
$$
y_i = f_i(z, y_1, \dots, y_N)
$$
for each $y_i = y_i(z)$, where $f_i$ is a polynomial and $N$ is the number of all stamps. Note that the degree $d_i$ of the polynomial $f_i$ does not exceed the maximal arity of the generators of the operad $P$. Then the elimination of the variables leads to an algebraic equation of degree at most $d = d_1^2 \dots d_N^2$ on $P(z)$. 

A couple of similar algorithms which in some cases reduce either the number or the degrees of the equations are also discussed in~\cite{kp}.

Knowing an algebraic equation for $P(z)$, one can evaluate the asymptotics for the coefficients $\dim P_n$ by well-known methods~\cite[Theorem~D]{flj-slg-fn}.

\section{The non-existence of a general algorithm}

	\label{sec:indecidable}

\begin{proposition}
	\label{prop:non-exist}
Suppose that the basic field $\KK$ is computable.

Consider the set $H_X$ of non-symmetric quadratic operads $P$ defined by a fixed finite set $X$ of binary generators and some finite set $R$ of quadratic relations on $X$.  Then there is a natural $n$ such that if $|X| \ge n$ then

(i) the  set of Wilf classes of operads from $H_X$ is infinite;

(ii) there does not exist an algorithm which takes as an input two sets $R_1, R_2$ of relations of two operads $P_1,P_2 \in H_X$ which returns {\em TRUE} if the operads belong to the same Wilf class and {\em FALSE} if not.	
	\end{proposition}

\begin{proof}
	We use Part (ii) of \cite[Theorem~3.1]{piontkovski2017growth}. 
	It states that, under the conditions of Proposition,  for some rational function $Q(z)$ there does not exist an algorithm which takes as an input the list $R$ of relations of the operad $P$ such that there is a coefficient-wise inequality $G_P(z)\le Q(z)$ and returns {\em TRUE} if 
	the equality $G_P(z)= Q(z)$ holds and {\em FALSE} if not.
	It follows that the set $H_X$ contains both operads with $G_P(z) = Q(z)$
   and operads with $G_P(z) \ne Q(z)$, and there is no algorithm to separate these two subsets. If $R_1$ is an operad of the first kind and $R_2$ is an operad of the second kind, then there is no general  algorithm to check whether they belong to the same Wilf class. This proves (ii).  
   Part (i) obliviously follows from (ii).
	\end{proof}

The rest of our results are positive.

\section{Operads, tree pattern avoidance, and unambiguous context-free languages}

\label{sec:operads_cf}

In this section, we prove Theorem~\ref{th:operads_cf_intro}.


First, let us recall the notation concerning operads and trees.

We consider planar rooted trees with finite possible type of vertices. These types are the following: the root (as a vertex of a special type) and a finite set $X$ of types for the  internal vertices and the leaves such that the vertices of the same type have the same number of children. Then the set $X$ is decomposed into the disjoint union 
$X  = X_0 \cup \dots \cup X_d$ for some $d>0$, where $X_i$ is the set of the types of vertices with $i$ children (the leaves are assumed to have zero children). We fix some type of leaves $x \in X_0$ and refer to the leaves of type $x$ as {\em free ends}. Below we consider trees with no leaves but with  free ends, so that we assume that $X_0 = \{x\}$ is a singleton.
 We call such trees {\em labelled trees} or simply {\em trees} (with the set of labels $X = X_0 \cup \dots \cup X_d$). 

One can {\em graft} (compose) trees by attaching the root of one tree to a free end of another one or replacing free ends with variables. 
Let us fix a (finite) set $Y$ of labelled trees called {\em patterns}.  We say that a tree $T$ {\em avoids} the pattern set $Y$ if there is no way to obtain $T$ by a subsequent grafting of  several trees to each other and at least one of them is a pattern. In other words, a tree does {\em not} avoid the patterns if it contains a subtree isomorphic to some element of $Y$. The problem is to enumerate all the trees avoiding the patterns. 

Some cases of this problem have been discussed in a number of papers. 
The case $X=X_0\cup X_2, X_0 = \{x \} $ of binary trees   has been considered by Loday~\cite{loday2005inversion} and \cite{rowland2010pattern} (with $X_2 = \{m \}$). The ternary tree case has been discussed in~\cite{gabriel2012pattern}.
The case of quadratic patterns in binary trees has been under consideration in~\cite{parker1993combinatorics}. The general labelled trees case  has been considered in~\cite[Section~2]{kp}, see also~\cite{giraudo2020tree}.

In  Polish notation, each such tree can be encoded by a  word on the alphabet $X$.  
So, all labelled trees avoiding the pattern set $Y$ are in a one-to-one correspondence to some formal language on the alphabet $X$. We denote this language as $L(X|Y)$. The language $L(X | \emptyset)$ is referred to as {\em free} and is denoted by $F_X$. 

For example, 
let $X=X_0\cup X_2$, where $X_2 = \{m_1, \dots, m_s \}$ and $X_0 = \{x\}$ (where $x$ is a mark for free end). Suppose that 
$ Y= Y_1 \cup Y_2$, where $ Y_1 =  \{ m_i x m_j xx | i, j =1..s \}$ and $Y_2$ is some set of trees which are not divisible by the elements of $Y_1$. 
The last condition means that any right-sided branch cannot have length 2 or more.  
Then the elements of $Y_2$ should have the form $wx...x$, where $w\in X_2^*$ 
is a word on the alphabet $X_2$ and the number of $x$-s is $\length (w)+1$. 
So,
$$
L(X|Y) =\{ u x^{\length (u)+1} | u\in X_2^* \mbox{and  no subword of $u$ belongs to }Y_2 \}.
$$
(Note that a word $v$ is called a subword of a word $u$ if $u = avb$ for some words $a$ and $b$.)
This means that the words of the language $L(X|Y)$ are in a one-to-one correspondence with the words on $X_2$ which have no subwords lying in $Y_2$, that is, with the monomial basis of the monomial associative algebra $\KK\la X_2 \ra / (Y_2) $. 

On the other hand, the words of the free language $L(X|\emptyset)$ with $X$ as above
are in a one-to-one correspondence with the generalized Dyck language with $s$ pairs of parentheses.

We will prove the following. 

\begin{theorem}
	\label{th:det_CF}
	Suppose that the sets $X$ and $Y$ as above are finite. Then  the language $L(X|Y)$ is deterministic context-free.
\end{theorem}

Given an alphabet $X$, one can associate to it a 
weight function $w: X^* \to {\Zpl}^s$ by assigning nonzero weights $w(a)
\in {\Zpl}^s$ to each letter $a\in X$ and expanding the weight to $X^*$
by the rule $w(uv) = w(u)+w(v)$.  For a language $L$ on $X$, one can consider the generating function $$
H_L(z) = \sum_{u\in L} z^{w(u)}, 
$$
where $z = (z_1, \dots, z_s )$  and  $z^{(n_1, \dots , n_s)} = z_1^{n_1} \dots z_s^{n_s} $. For example, in the case $w(x_1) = \dots =w(x_s) = 1 \in \Zpl$ the formal power series $H_L(z_1)$ is the generating function for the growth function $g_L(n) = \# \{ u\in L | \length(u) = n \}$ of the language $L$.

The famous enumeration theorem by Chomsky and Sch{\"u}tzenberger \cite{chomsky1963algebraic} 
describes growth functions of unambiguous context-free languages. Using this and the theorem by
D’Alessandro, Intrigila, and Varricchio about generating function of sparse context free languages~\cite{d2006structure}, we get

\begin{corollary}
	\label{cor:trees_intro}
	Let $H(z)$ be the generating function of the language $L(X|Y)$  above, where the sets $X$ and $Y$ are finite and $z = (z_1, \dots, z_s)$ is a vector of variables. Then the formal power series $H(z)$ satisfies a non-trivial algebraic equation with coefficients in $\Z[z]$. If, moreover, the growth of the language is sub-exponential, then the function $H(z)$ is rational.
\end{corollary}

The Chomsky--Sch{\"u}tzenberger theorem gives a way to construct a system of algebraic equations for $H(z)$. Its variables are the generating functions of the sub-languages of $L(X|Y)$
which can be derived from the non-terminals of the unambiguous context-free grammar. If we use the grammar $G$ for the language 
$L(X|Y)$ (see Lemma~\ref{lem:grammars_are_unamb}), we get a system equivalent to the one described in Subsection~\ref{sec:we} and~\cite[2.2.1]{kp}. Moreover, after a triangular linear change of variables, it is also equivalent to another system of equations described in~\cite[2.2.2]{kp}. 
In the case of binary one-relator operads, these two kinds of systems were created earlier by Rowland (see~\cite{rowland2010pattern}; we discuss this case in  Section~\ref{sec:one_rel_rowland} below).


\begin{corollary}
	\label{prop:operad_cf}
		\label{cor:operad_cf}
	Let $P$ be a finitely generated non-symmetric  operad having a finite Groebner basis of relations. 
	Then its natural monomial basis forms a deterministic context-free language
	$L(X|Y)$ for some finite $X$ and $Y$.  In particular, the generating series of the operad satisfy the conclusion of Corollary~\ref{cor:trees_intro}.
\end{corollary}

%
%

Given a tree $t\in F_X$, its {\em height} $\height  t$ is the maximal number of internal nodes lying on the same branch. Let $d = \max \{ \height t | t\in Y\}$ be the maximal pattern height. In the case of empty $Y$, we put $d=0$.

We say that a tree $t$ is a {\em rooted subtree} of a tree $v$ (notation: $t\reqtree v$) if $v$ can be obtained from $t$ by grafting some other trees onto it.  If, in addition, $t\ne v$, we write $r \rtree t$. 

In the notation of Theorem~\ref{th:det_CF}, let us denote $L = L(X|Y)$ and $L' = L'(X|Y)$. For $n\ge 0$, let $L_n$ (resp., $L'_n$) denote the set of all trees in $L$ (resp., $L'$) having a height of at most $n$.

Let $t \in L$. Let $\widehat {M_t} = \{ v \in L | t \reqtree v\}$ denotes the set of all the trees of $L$ obtained from $t$ by grafting other trees onto it. Put 
$$
M_t = \widehat {M_t} \setminus  \bigcup_{s\in L_d: t\rtree s} \widehat {M_{s}}.
$$

\begin{lemma}
	The  language $L = L(X|Y)$ is the disjoint union of the subsets $M_t$ with $t\in L_d$.
\end{lemma}

\begin{proof}
	Obviously, the union of all such sets $M_t$ is the same as the union of all sets $\widehat {M_t}$, where $t$ runs $L_d$. Since $\widehat {M_1} =L$ (where $1$ is the tree consisting of the root and single free end), this union is equal to $L$. Let us prove that the union is disjoint. 
	
	Ad absurdum, suppose that for some different $s, t \in L_d$ there exists a tree $p \in M_s \cap M_t$.
	Since $p\in \widehat {M_{s}} \cap \widehat {M_t}$, the both trees $s$ and $t$ are rooted subtrees of $p$. So, there is  the minimal 	(w.~r.~t. the relation "$\rtree $") rooted subtree $r$ of $p$ such that both $s$ and $t$ are rooted subtrees of $r$.  The set of internal nodes of $r$ (as a subgraph of $p$) is the union of the sets of internal nodes of $s$ and $t$, so that $\height r \le \max\{\height s, \height t\} \le d$. Since $r$ is a rooted subtree of $p\in L$, it follows that $r\in L$, so that $r\in L_d$. Then $p\in \widehat M_r$. If $t\ne p$ (or, respectively, $s\ne p$), then 
	$$
	p\in  M_t \subset \widehat {M_t} \setminus \widehat {M_r}
	$$
	(resp., $p\in \widehat {M_s} \setminus \widehat {M_r}$), in contradiction to the condition $p\in \widehat M_r$. So, $p=s=t$: this contradicts the choice of $s$ and $t$.
\end{proof}

Now, let us define a context-free grammar $G$ for the languages $L$ as follows. Let the sets of terminal symbols be $X$,
and let $V = \{ T_v |v\in L_d \}\cup \{S \}$ be the set of non-terminal symbols. The sets of rules of these grammars are the following.
First,  for each $v\in L_d$, let $m=m_v$ be the label of the root vertex of $v$, and let 
$k$ be the number of children of this vertex (so that $m\in X_k$). If $k\ge 1$ then the rule 
$$
T_v \to m T_{v_1} \dots T_{v_k}
$$
exists  for some $v_1, \dots, v_k \in L_d$ iff $m v_1 \dots v_k\in M_v$.
Next,  there is a rule 
$$
T_x \to x.
$$ 
The initial rules are 
$$
S\to T_v
$$ 
for all $v\in L_d$.

\begin{lemma}
	\label{lem:grammars_are_unamb}	
	The  grammar  $G$ is unambiguous  
	and  generates the languages $L= L(X|Y)$.
\end{lemma}

\begin{proof}
	
	Let us prove that for each word $w$ from $L$ there exists a unique rightmost derivation in the grammar $G$. 
	
	Let us first show that for each such word $w \in M_v$ there is a rightmost derivation beginning with $S\to T_v$.
	Indeed, if $w =x$, then there is  unique derivation $S\to T_x\to x$. Otherwise, there are unique $v, v_1, \dots , v_k$, and $m$ such that 
	$w = m v_1 \dots v_k\in M_v$. By the induction argument, for each $v_i$ there is a unique rightmost derivation. It  has the form 
	$$S \to T_{h_i}\to \dots \to   v_i , $$
	where $v_i \in M_{h_i}$. 
	Then there is a rightmost derivation 
	\begin{multline}
		\label{eq:der_w}
		S \to T_v \to m T_{h_1} \dots T_{h_{k-1}} T_{h_k} \to  \dots
		m T_{h_1} \dots T_{h_{k-1}} v_k \\
		\to \dots \to
		m T_{h_1}  \dots {v_{k-1}} v_k  
		\\
		\to \dots \to m v_1 \dots v_k = w.
	\end{multline}

	Let us show that this derivation is unique.  By the induction argument, it is sufficient to show that the first step is unique. Moreover, we can assume by induction that the unique rightmost derivation  for any word $b$ of length less than the one $w$  begins with $S\to T_{c}$ where $b\in M_c$.
	
	So, assume that there is  a rightmost derivation $S \sto w$ with the initial step $S\to T_{v'}$; we need to show that $v=v'$.
	The next step of the derivation must be of  the form 
	$T_{v'} \to m' T_{v'_1} \dots T_{v'_{k'}}$, where 
	$m' {v'_1} \dots {v'_{k'}} \in M_{v'}$. It follows that 
	$w = m' w'_1 \dots w'_{k'}$, where for each $i$ there exists a (rightmost) derivation $T_{v'_i} \sto {w'_i}$. Here 
	$m'$ corresponds to the mark of the root of the tree $w$ and 
	the subwords $w'_1 , \dots , w'_{k'}$ correspond to its branches. It follows that  $m'=m$, $k'=k$, and $w'_i = w_i$ for each $i =1, \dots ,k$. Moreover, since for each $w_i$ there are two rightmost derivations 
	$$
	S \to T_{v'_i} \sto {w'_i} = w_i
	$$
	and 
	$$
	S \to T_{v_i} \sto w_i,
	$$
	we use the induction assumption to conclude that $v'_i = v_i$. Thus, the word $ m {v_1} \dots {v_{k}} = m' {v'_1} \dots {v'_{k'}} $ belongs to both $M_v$ and $M_{v'}$, so that $v=v'$. 
\end{proof}

\begin{lemma}
	\label{lem:grammars_are_determ}
	The grammar $G$ is deterministic.
\end{lemma}

\begin{proof}
	We have to show that each word $a$ appearing in the derivation process for some $w\in L$ (respectively, $w\in L'$) has a forced handle. If $w\in X$, then both words appearing in the derivation $S\to T_w\to w$ obviously coincide with their one-symbol handles, so that the handles are forced. 
	
	It remains to show that each word $a$ appearing in the 
	derivation~(\ref{eq:der_w}) 
	for a word $w$ of length at least 2 
	has a forced handle.  Let $r$ be the maximal prefix of $a$ which ends with a non-terminal, so that $a=rs$
	where $r\in {X\cup V}^*V \cup \{ 1 \}$ 
	and $s \in X^*$ consists of variables.
	
	On the other side, the word $a$ has the form $a = uhq$, where $u,h \in (X\cup V)^*$, $h$
	is the handle, and $q\in X^*$ is a word consisting of terminal symbols. The handle $h$ of $a$ is  the right-hand-side of some grammar rule, so that it has either the form $T_v$ for some $v\in L_d$, or the form $\mu T_{\alpha_1} \dots T_{\alpha_s}$ for some $\mu,{\alpha_1}, \dots , {\alpha_s}$, or is equal to $x$.

	Let us compare both decompositions $rs$ and $ahq$ of the same word $u$.
	Obviously, if $h=T_v$ for some $v\ne x$, then $a=T_v$ and this is the first step $S\to T_v$ of the derivation. Still, if $a = T_v$ for some $v\in L$, then $h=a$ is forced, because no word but $a$ derived from $S$ can begin with $T_v$.  We will not consider this case further.  
	Next, if the handle $h$ is of the form $\mu T_{\alpha_1} \dots T_{\alpha_s}$, then it consists of the rightmost sequence of non-terminals in $uh$ accompanied with the preceding terminal symbol, so that 
	$r=uh$ and $s=q$. Finally, if $h = x$, then $s = s_1 x q$, where $s_1$ does not contain $x$, so that $s_1\in (X\setminus X_0)^*$.

	Let $t \in [1,k]$ be the maximal number such that $a$ appears before the $(k+1)$-symbol initial segment of the word $m T_{h_1} \dots T_{h_k}$ in~(\ref{eq:der_w}) is changed, so that the derivation of $w$ splits as 
	\begin{multline*}
		S \to T_v \to m T_{h_1} \dots  T_{h_k} \to  \dots
		m T_{h_1} \dots T_{h_{k-1}} v_k \\
		\to \dots \to 
		m T_{h_1} \dots  T_{h_{t-1}} T_{h_{t}} v_{t+1} \dots v_k \\ 
		\to \dots \to a
		\to \dots m T_{h_1} \dots  T_{h_{t-1}} v_{{t}} v_{t+1} \dots v_k 
		\\
		\to \dots  \to m v_1 \dots v_k = w.
	\end{multline*}
	
	Then $a$ has the form $a= m T_{h_1} \dots  T_{h_{t-1}} a'$ with $a' = a'' v_{t+1} \dots v_k$, where 
	the word $a''$ appears in the derivation 
	\begin{equation}
		\label{eq:derivation_vt}
		S\to T_{h_{t}}\to \dots a'' \to \dots \to v_t
	\end{equation}
	of the word $v_t$. 
	
	Ad absurdum, suppose that the rightmost derivation of some word $\tilde w \in L$ (or $\tilde w \in L'$, in the case of grammar $G'$) contains a word $\tilde a = uh \tilde q$  (where $\tilde q \in X^*$) which has another handle $\tilde h$. If both handles $h$ and $\tilde h$ are of the form $\mu T_{\alpha_1} \dots T_{\alpha_s}$, then each of them is uniquely defined as the rightmost subword of $u$ having this form, so that $\tilde h=h$, a contradiction. Now, suppose that both handles  $h$ and $\tilde h$ are single symbols $x$. Then  the handles are uniquely defined as the leftmost occurrences $x$  in the words $uhq$ and, respectively, $uh \tilde q$.
	Then we again get $\tilde h=h$, a contradiction.
	
	Now, it remains to consider the case when both handles are of different kinds. We can assume that $h = \mu T_{\alpha_1} \dots T_{\alpha_s}$ and $\tilde h =x$, so that $a = uhs$ and $\tilde a = uh\tilde s_1 \tilde h \tilde q$. 
	Then $h$ is the handle of the word $a' = u'hs'$ 
	appearing in the derivation~(\ref{eq:derivation_vt})  of the word $v_t$,
	where $s = s'v_{t+1} \dots v_k$.
	On the other hand, since the initial symbol of the word $\tilde w$ is $m$, then $\tilde w = m \tilde v_1 \dots \tilde v_k$ for some $\tilde v_1, \dots,  \tilde v_k$. Then the derivation of $\tilde w $ 
	has the following form similar to (\ref{eq:der_w})
	\begin{multline*}
		S \to T_{\tilde v} \to m T_{\tilde h_1} \dots  T_{\tilde h_k} \to  \dots \\
		\to
		m T_{\tilde h_1} \dots T_{\tilde h_{k-1}} \tilde v_k 
		\to \dots \to 
		m T_{\tilde h_1} \dots   T_{\tilde h_{t-1}} T_{\tilde h_{t}} \tilde v_{t+1}\dots \tilde v_k  \\
		\to  \dots \to \tilde a  \to \dots  \to
		m T_{\tilde h_1} \dots   T_{\tilde h_{t-1}} \tilde v_{{t}} \tilde v_{t+1}\dots \tilde v_k \\
		\to \dots 
		\to m \tilde v_1 \dots \tilde v_k = \tilde w.
	\end{multline*}
	Here $\tilde h_1 = h_1, \dots, \tilde h_{t-1} = h_{t-1}$, because the initial segments $uh$ of the words $a$ and $\tilde a$ coincide. Then $\tilde h$ is the handle of the word $\tilde a'' = u'h \tilde s'$ in the induced derivation
	\begin{equation*}
		S\to T_{\tilde h_{t}}\to \dots \to  \tilde a'' \to \dots \to \tilde v_t
	\end{equation*}
	of the word $\tilde v_t$. By induction, the handle $h $ is forced in the word $a'' = u'hs'$. This means that $h$ must coincide with any handle in a word of the form $u'h g$ for any $g\in X^*$. So, the handle $\tilde h$ of the word $\tilde a'' = u'h \tilde s' $
	must coincide with $h$, a contradiction.  
\end{proof}

\begin{proof}[Proof of Theorem~\ref{th:det_CF}]
	It  follows from Lemmata~\ref{lem:grammars_are_unamb} and \ref{lem:grammars_are_determ} that the language $L(X|Y)$ is defined by deterministic context-free grammars. Since this language is a  subset of the prefix-free language $L(X|\emptyset)$, it is prefix-free. So, the language is deterministic context-free. 
\end{proof}

\section{One-relator binary operads and tree pattern avoidance}
\label{sec:one_rel_rowland}


Consider an  operad $P$ with a single  binary generator. Then, all internal vertices of the monomials of the operad have the same label (corresponding to the  generator), so, these monomials can be enumerated by planar binary trees with no labels. 
Let the operad have a single monomial relation $t$.
Then the monomial linear basis of the operad is in natural bijection with the 
binary trees without labels avoiding 
pattern $t$.

For example, the single relation $m(m(x_i,x_j),x_k)) - m(x_i,m(x_j,x_k))$ of the non-symmetric associativity operad $\mathrm{Ass}$ (see Subsection~\ref{sec: associativity_oper}) looks as a difference of two trees, see Figure~\ref{fig:assocrel}. 
\begin{figure}[ht]
\includegraphics[scale=0.06]{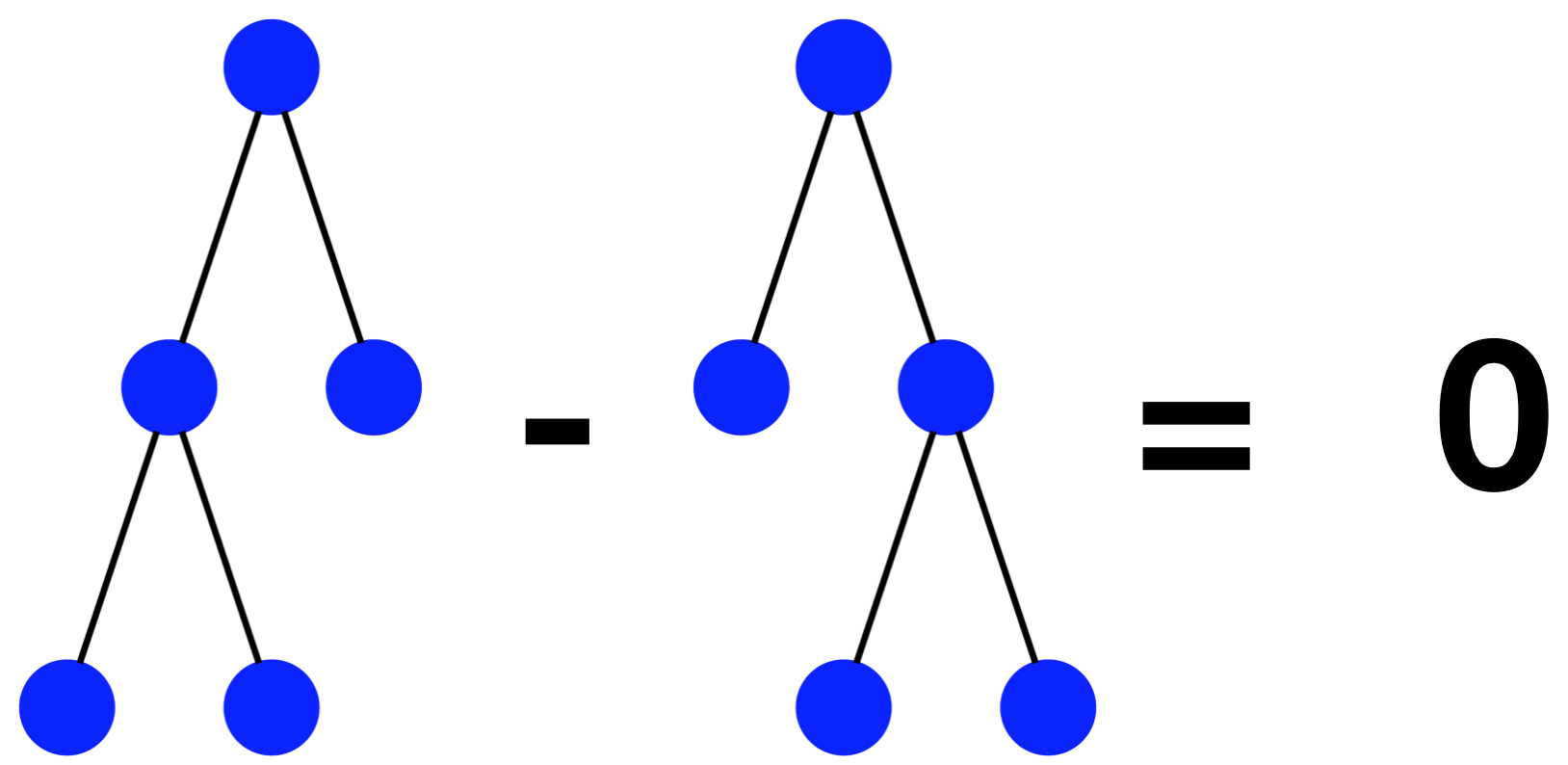}
\caption{The associativity relation}
\label{fig:assocrel}
\end{figure}
 The leading monomial of the relation is the monomial $t$ corresponding to the tree in Figure~\ref{fig:simple_monom}.  
\begin{figure}[ht]
\includegraphics[scale=0.04]{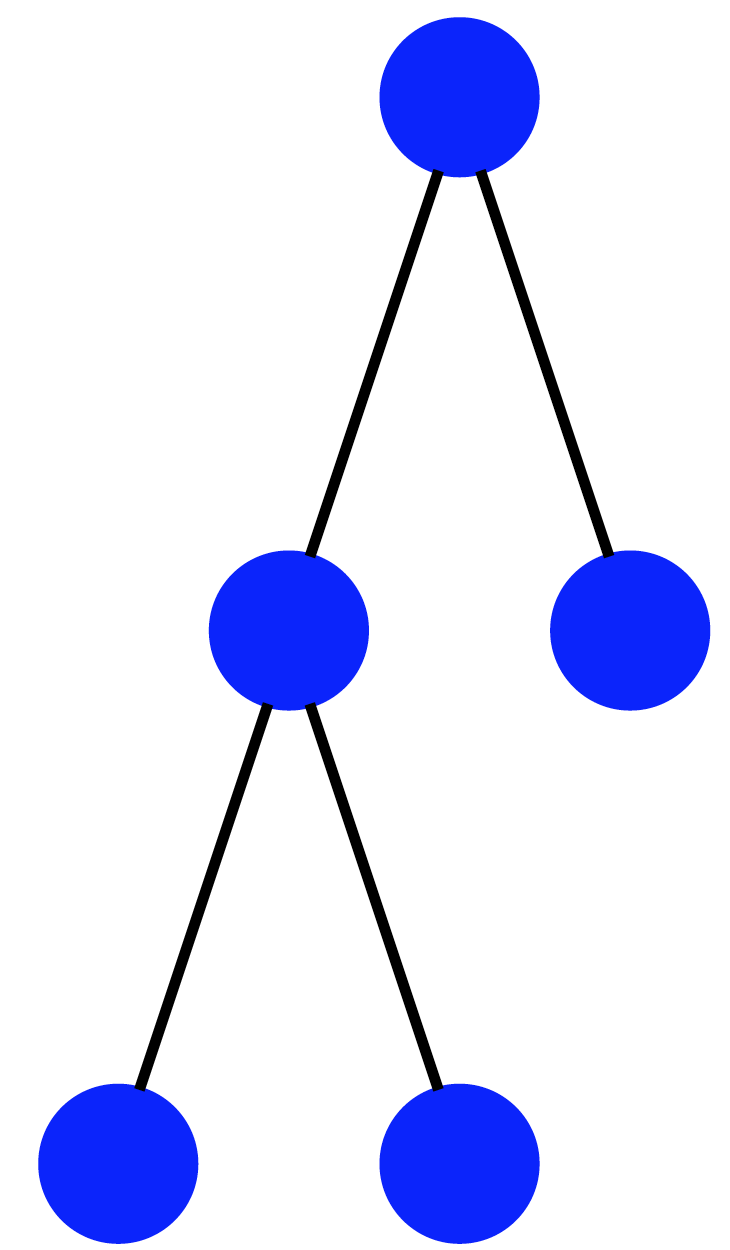}
\caption{The leading monomial ($t$)}
\label{fig:simple_monom}
\end{figure}
 Then a linear basis of each component $\mathrm{Ass}_n$ of the operad consists of the
monomials avoiding $t$, that is, the ones with the trees in Figure~\ref{fig:ass_n}.
\begin{figure}[ht]
\includegraphics[scale=0.07]{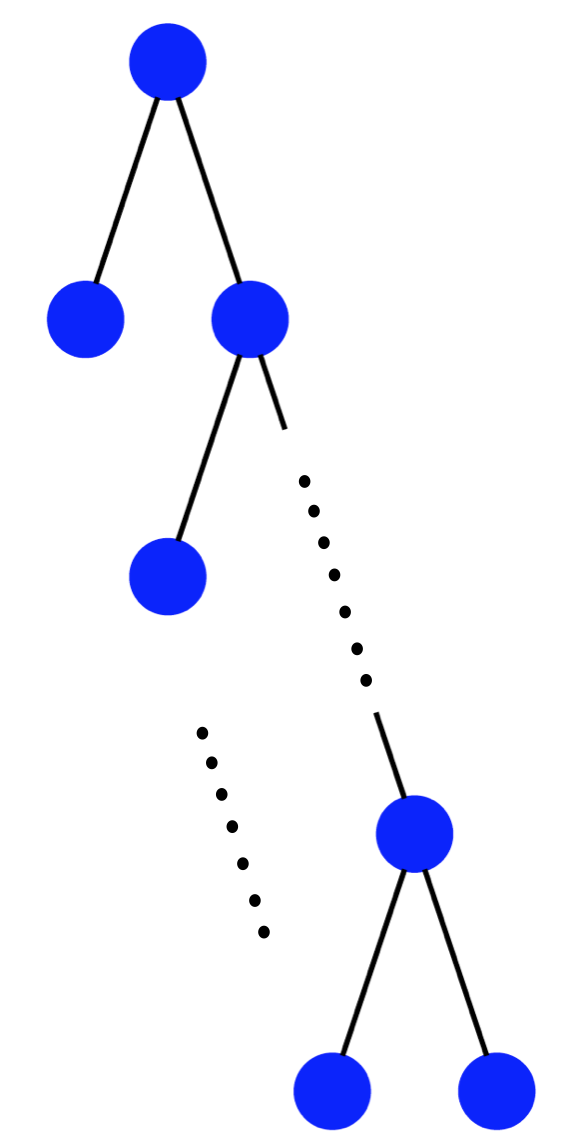}
\caption{A monomial avoiding $t$}
\label{fig:ass_n}
\end{figure}

The methods of enumeration of such trees and the Wilf classes has been introduced by Rowland~\cite{rowland2010pattern}; here we recall some of his results. 
Let $a_n$ denote the number of words with $n$ vertices (both internal and external) which avoids 
pattern $t$. The generating function 
$$
Av_t(x) = \sum_{n\ge 1} a_n x^n
$$
is called the {\em avoidance } function of $t$. Two patterns are Wilf equivalent (or avoidance equivalent, in terms of~\cite{rowland2010pattern}) if their avoidance functions are equal. Obviously, Wilf equivalent patterns have the same number of leaves. 

Note that the binary tree with $n$ vertices has $(n-1)/2$ leaves. 
Since the $k$-th component $P_k$ of the operad is spanned by 
trees with $k$ leaves, the generation function of the operad is 
$$
G_P(z) = \sum_n a_n z^{(n-1)/2} = Av_t(\sqrt{z}) /\sqrt{z}. 
$$
The Wilf classes of these operads correspond to the 
Wilf classes of the patterns. 

Now, let $a_{n,k}$ be the number of binary trees with $n$ vertices which contains exactly $k$ copies of 
pattern $t$. The {\em enumerating generating function} of $t$ is 
$$
En_t(x,y) = \sum_{n\ge 1, k\ge 0} a_{n,k} x^n y^k.
$$
Two patterns are enumerating equivalent if their  enumerating generating functions are equal. As $Av_t(x) = En_t(x,0)$, this is a stronger version of Wilf equivalence.

The next conjecture states that the strong and weak version of Wilf classes should coincide.  
 
\begin{conjecture}[Rowland]
	\label{conj:rowland}
	If two patterns $s$ and $t$ are Wilf equivalent, then they are enumerating equivalent.
	\end{conjecture} 

In the next section, we discuss a  homological interpretation of the enumerating generating function and the above Rowland conjecture. 

In~\cite{rowland2010pattern}, the algebraic equations defining the functions $Av_t(x)$ and $
En_t(x,y)$ are listed for all patterns with at most 6 leaves. Moreover, the number $A(n)$ of Wilf classes of the $n$-leaves patterns are calculated up to $n=8$. These numbers are 
$1,1,1,2,3,7,15,44$. 
Moreover, the calculations for $n\leq 7$ confirm Conjecture~\ref{conj:rowland}.

Note that the sequence $A(n)$ is listed in 
Sloane's  On-Line Encyclopedia of Integer Sequences as A161746, see https://oeis.org/A161746.

Now, we have tried to take the next step toward the enumeration of Wilf classes of one-relator binary operads.

First, we have provided a calculation similar to Rowland's. 
We have used the methods from~\cite{rowland2010pattern} to construct 
the systems of algebraic equations.
The systems of equations were generated by an { ad hoc} 
C\# software. We then proceeded with  
the {\em Wolfram Mathematica} elimination of variables procedure. 
As both Groebner bases methods and computer performance have been improved within the last decade, we hoped to calculate the number $A(n)$ of Wilf classes $A(n)$ and the number $E(n)$ of enumeration classes for the next values of $n$. However, we have done no more than confirm Rowland's calculation, that is, to list 15 equations 
for $En(x,y)$ for $n=7$ and 44 equations for $Av(x)$ for $n=8$.

The systems of algebraic equations over the polynomials give recurrent equations for the coefficients of the power series solutions. As our systems are of a rather special kind, these recurrent equations are rather simple in our case. 
Following this, we developed a Python package (called {\em Friend reduce}, as we have parallelized the calculations among our friends' laptops)
which, for each of our systems of algebraic equations, finds its solution in the form of the collection of  truncated formal power series. Given a number $n$ of the leaves, for each pattern $t$ with $n$ leaves we have, particularly,  calculated all formal power series $Av_t(x)$ and $En_t(x)$ truncated up to some power $x^k$. The numbers $\overline A(n)$ and $\overline E(n)$ of such truncated series give lower bounds for the numbers $A(n)$
and $E(n)$.

Both packages can be downloaded at \\ https://github.com/atcherkasov/tree-Wilf-classes.

The results of the calculations are presented in Table~\ref{tab:Wilf_calculations_results}.

\begin{table}
	\caption{Lower bounds for the numbers Wilf classes of one-relator quadratic operads (the generating functions are calculated up to $o(x^k)$)}
	\label{tab:Wilf_calculations_results}
	\begin{tabular}{lccccc}
		\toprule
		$n$&8&9 & 10 & 11 &12 \\
		\midrule
		$A(n) = $   & 43 & 	 & &  &  \\
		$A(n) \ge $   & & 	$ 136$ & $ 458$ & $ 1662$ & $ 6096$ \\
		$k$ & 257 &	257 &	257	& 257 &	201\\
		\midrule
		$E(n)  \ge $   & 43 & 	$ 136$ & $ 458$ & $ 1662$ & $ 6096$ \\
		$k$ & 257 &	257 &	257	& 201 &	157\\
		\bottomrule
	\end{tabular}
\end{table}

Note that for $n=8$, we give the exact value $A(8)=43$ in place of the lower bound.  Moreover, this equality contradicts the value $44$ of $A(8)$ listed in~\cite[p.~756]{rowland2010pattern}. Nevertheless, we are sure of this equality for the following reasons. After the elimination of variables, we get 44 different equations for $A_t(x)$ with 8-leave patterns $t$. 
Still, one can show that two of these equations define the same algebraic power series. These are the equations 
\begin{multline}
	\label{eq:1}
(-x+x^3-x^5-x^7+x^9+x^{13})\\
+G\cdot (1-2 x^2+4  x^4+2  x^6-6  x^8+2  x^{10}-4  x^{12}+3  x^{14})\\
+G^2\cdot (-3
 x^3+9x^7-8  x^9+7  x^{11}-8  x^{13}+3  x^{15})\\
+ G^3 \cdot (-3  x^6+7  x^8-8  x^{10}+6  x^{12}-3  x^{14}+x^{16})\\
+G^4\cdot (3  x^9-2  x^{11}-3  x^{13}+2  x^{15})+G^5\cdot (2  x^{12}-3  x^{14}+x^{16})=0	
\end{multline}
and
\begin{multline}
	\label{eq:2}
(x-x^3+2  x^5+x^9)+G\cdot (-1+2  x^2-6  x^4+2  x^6-3  x^8+2  x^{10})\\
+G^2\cdot (4  x^3-3  x^5+3  x^7-3  x^9+x^{11})\\
+G^3\cdot (-x^6+x^{10})+G^4\cdot (-2x^9+x^{11})=0,  
\end{multline}
where $G = Av(x)$. 
Indeed, the polynomial on the left hand side of~(\ref{eq:1})
is divisible by the one in~(\ref{eq:2}), as we have checked by {\em Wolfram Mathematica}.  Therefore, any  formal power series solution of~(\ref{eq:2}) is also a solution of~(\ref{eq:1}). As the solutions satisfying the initial conditions are unique for each equation, these solutions are the same. 
This gives the inequality $A(8) \le 43$.
Using the result $A(8) \ge 43$ obtained by the calculation of truncated series, we conclude that $A(8) = 43$.

We see that the number of different formal power series stabilizes at some values of $k$. That is why we establish

\begin{conjecture}
	The values of the numbers $A(n)$ and $E(n)$ of the Wilf classes and enumeration classes  
	of one-relator operads  with the relation of degree $n$ 
	for $n=8, 9,10,11,12$ are equal to the upper bounds listed in Table~\ref{tab:Wilf_calculations_results}.
\end{conjecture} 

This conjecture would imply that  Conjecture~\ref{conj:rowland} holds for all 
$n\le 12$.

\section{Wilf classes and homology}
\label{sec:homology}

The purpose of this section is to discuss a direction for future research. We believe that this research 
would give methods and algorithms to study refined versions of the Wilf classes based on homological invariants. 

In~\cite{dotsenko2013quillen} Dotsenko and Khoroshkin have introduced a differential graded resolution of a  shuffle monomial operad $P$ (moreover, the construction is generalized to the case of a general operad using  Groebner bases). The version of this construction for non-symmetric operads is described in~\cite[2.2.3]{kp}. It is a free operad generated by the generators of $P$ and additional trees that are in one-to-one correspondence with the monomials covered  by the monomial relations of $P$ in an ``indecomposable'' way.
Let $b_{k,n}$ denote the number of the degree $n$ generators 
containing exactly $k$ copies of the relations.
For some particular monomial operads (such as the quadratic ones and the operads defined by the patterns of Class 4.2 from~\cite{rowland2010pattern}), the above resolution is minimal, so that the numbers $b_{k,n}$ are equal to the Betti numbers $\beta_{k,n}$ of $P$ in the sense of Quillen homology (that is, the   number of the minimal generators of internal degree $n$ and homological degree $k$ in the minimal differential graded model of $P$). In general, we have useful inequalities $b_{n,k} \ge \beta_{n,k}$.  

Let $B_P(z,y) = \sum_{k,n} b_{k,n} z^n y^k$ be the  generating function. One can consider it as a version of the 
Poincare series of $P$. 
One can consider ``homological'' Wilf classes of the operads by 
saying that two operads are homologically equivalent if their 
functions $B_P(s,y)$ coincide. 

\begin{proposition}
	The functions $B_P(z,y)$ and $En_t(x,y)$ can be uniquely defined in terms of each other. Therefore, two  one-relator binary operads belong to the same homological Wilf class 
	if and only if their relations are enumeration equivalent.
\end{proposition}	

Then Rowland's Conjecture~\ref{conj:rowland} implies

\begin{conjecture}
	In the class of one-relator binary operads $P$, 
	the Betti numbers $b_{n,k}$ are uniquely defined by the 
	generating series $G_P(z)$.
\end{conjecture} 

The calculations described in Section~\ref{sec:one_rel_rowland}
confirm this conjecture provided that the degree of the operad relation is less or equal to 12. 
If the conjecture turns out to be true, it will be interesting to get an answer to the question: are Quillen  Betti numbers $\beta_{k,n}$ of one-relator binary operads also determined by the Wilf class?

For the direct calculation of Quillen homology using the irreducible elements of this resolution, one can use the simplicial homology (cf. chain complexes of simplices mentioned in Subsetcion 2.1 of~\cite{dotsenko2013quillen}). 
It is our future plan to apply computational topology software for evaluating the Poincare series of operads and to provide a stronger homological classification.

\bibliographystyle{ACM-Reference-Format}
\bibliography{biblio_operads}

\end{document}